\documentclass{amsart}
\usepackage{amsmath}
\usepackage{amssymb}
\usepackage{color}
\usepackage{amscd}
\input xy
\xyoption{all}
\newtheorem{Lemma}{Lemma}[section]
\newtheorem{remark}[Lemma]{Remark}

\newtheorem{theorem}[Lemma]{Theorem}
\newtheorem{lemma}[Lemma]{Lemma}
\newtheorem{proposition}[Lemma]{Proposition}
\newtheorem{corollary}[Lemma]{Corollary}

\newtheorem{example}[Lemma]{Example}

\newcommand{\Cal}[1]{{\mathcal #1}}
\newcommand{\End}{\operatorname{End}}
\newcommand{\Tor}{\operatorname{Tor}}

\newcommand{\Hom}{\operatorname{Hom}}

\newcommand{\Ext}{\operatorname{Ext}}

\newcommand{\Rtop}{\operatorname{-top}}

\newcommand{\Mod}{\operatorname{Mod-\!}}

\newcommand{\lMod}{\mbox{\rm -Mod}}

\DeclareMathOperator{\pdim}{p.dim}
\DeclareMathOperator{\Add}{Add}

\newcommand{\cmat}{\left(\begin{array}}
\newcommand{\fmat}{\end{array}\right)}

\newcommand{\Z}{\mathbb{Z}}

\begin{document}
   \title{Covering classes, strongly flat modules, and completions}
  \author[Alberto Facchini]{Alberto Facchini}
\address{Dipartimento di Matematica, Universit\`a di Padova, 35121 Padova, Italy}
 \email{facchini@math.unipd.it}
\thanks{The first author was partially supported by Dipartimento di Matematica ``Tullio Levi-Civita'' of Universit\`a di Padova (Project BIRD163492/16 ``Categorical homological methods in the study of algebraic structures'' and Research program DOR1714214 ``Anelli e categorie di moduli''). The second author was supported by a grant from  IPM}
 \author[Zahra Nazemian]{Zahra Nazemian}
\address{School of Mathematics, Institute for Research in Fundamental Sciences (IPM),
	P. O. Box: 19395-5746, Tehran, Iran}
 \email{z\_nazemian@yahoo.com}
   \keywords{Covering class, Strongly flat module, Completion, Cotorsion module, $R$-topology. \\ \protect \indent 2010 {\it Mathematics Subject Classification.} Primary 16E30, 16W80.
Secondary 18G15.}
      \begin{abstract}  We study some closely interrelated notions of Homological Algebra: (1) We define a topology on modules over a not-necessarily commutative ring $R$ that coincides with the $R$-topology
 defined by 
Matlis when $R$ is commutative. (2) We consider the class $ \mathcal{SF}$ of strongly flat modules when $R$ is a right Ore domain with classical right quotient ring $Q$. Strongly flat modules are flat. The completion of $R$ in its $R$-topology is a strongly flat $R$-module. (3) We consider some results related to the question whether 
$ \mathcal{SF}$ a covering class implies $ \mathcal{SF}$ closed under direct limit. 
This is a particular case of the so-called Enochs' Conjecture (whether covering classes are closed
 under direct limit). 
 
 Some of our results concerns right chain domains. For instance, we show
that if the class of strongly flat modules  over a right chain domain $R$ is 
 covering, then 
$R$ is right invariant. In this case, flat $R$-modules are strongly flat.
\end{abstract}

    \maketitle

\section{Introduction}

The aim of this paper is to highlight some relations between completions, strongly flat modules and perfect rings in the non-commutative case. We explore some connections between some notions of Homological Algebra (cotorsion modules) and topological rings (completions in some natural topologies). These connections are well known for modules over commutative rings, thanks to Matlis, who proved that the completion in the $R$-toplogy for an integral domain $R$ is closely related to the cotorsion completion functor $\Ext_R^1(K,-)$. Here $Q$ is the field of fractions of $R$ and $K:=Q/R$. We investigate these connections in the non-commutative case, defining a suitable $R$-topology on any module over a not-necessarily commutative ring $R$. This leads us to the study of strongly flat modules, because the completion of $R$ in its $R$-topology turns out to be a strongly flat $R$-module (Theorem~\ref{4.9}).

We consider strongly flat modules over non-commutative rings as defined in \cite[Section~3]{submitted}.
The class of strongly flat modules lies between the class of projective modules and the class of flat modules.  
In particular, we study when the class of strongly flat modules is covering, because this is related to an open problem posed by Enochs, whether ``every covering class is closed under direct limit" (see for example \cite [Open problem 5.4] {approx}). 
Since flat modules are direct limits of projective modules, the class of strongly flat modules is closed under direct limits if and only if flat modules are strongly flat. 
 Bazzoni and Salce  \cite{silsal} gave a complete answer to this question for modules  over commutative domains, completely determining when  the class of strongly flat modules  over a commutative domain  is covering. 
Subsequently,  Bazzoni and Positselski generalized this to arbitrary commutative rings in \cite{Leonidsilvana}. They proved that, for a commutative ring $R$, 
the class $ \mathcal{SF}$ of strongly flat  modules is covering   if and only if 
flat modules are strongly flat, if and only if 
$R/aR$ is a perfect ring for every  regular element $a \in R$. In our Example~\ref{5.18}, we will show that there exist non-invariant chain domains $R$ for which $\End(R/I)$ is perfect for every non-zero principal right or left ideal
$I$  of $R$, but the class of strongly flat left $R$-modules is not covering. Very recent papers related to these topics are the articles \cite{BP,P}
by Bazzoni and Positselski.

For a commutative ring $R$, the set of regular elements is always an Ore set, and if $Q$ denotes the classical quotient ring of $R$,  the class of strongly flat modules is $^\bot \{Q^\bot \} $ \cite{ FSALCE}. 
The  generalization of strongly flat modules to non-commutative rings given in \cite {submitted} depends on the choice of the overring $Q$ of $R$. More precisely, if 
$\varphi \colon R\to Q$  is a bimorphism in the category of rings, that is, $\varphi$ is both a monomorphism and an epimorphism, we assume that ${}_RQ$ is a flat left $R$-module. We view at $R$ as a subring of $Q$ and $\varphi \colon R\to Q$ as the inclusion.
Then a  left $R$-module $_RM$ is {\em Matlis-cotorsion} if $\Ext^1(_RQ,{}_RM)=0$ \cite{submitted}.
Let $\mathcal {MC}$ denote the  class of Matlis-cotorsion left $R$-modules. For any  class of left $R$-modules $ \mathcal {A}$, set $ ^ \bot \mathcal {A}:= \{ \,B\in R\lMod  \mid \Ext_R ^1 (B, A) = 0$ for every $A \in \mathcal {A}\,\}$ and $\mathcal {A}^ \bot := \{ \,B\in R\lMod  \mid\Ext_R ^1 (A, B) = 0$ for every $A \in \mathcal {A}\,\}$. 
A left $R$-module is {\em strongly flat} if it is in  $ ^\bot \mathcal {MC} $.
The class of strongly flat left $R$-modules will be denoted by $\mathcal {SF}$.
By
\cite  [Theorem 6.11] {approx},  the cotorsion pair $(\mathcal{SF} , \mathcal{MC})$ is
complete, that is, every left $R$-module has a special
$\mathcal{MC}$-preenvelope (or, equivalently, every left module has a special $\mathcal{SF}$-precover).
Thus, by \cite  [Corollary  6.13] {approx}, the class $\mathcal{SF} $ consists of
all direct summands of
modules $N$ such that $N$ fits into an exact sequence of the form
$$ 0  \to F \to N \to G \to 0,$$
where $F$ is a free $R$-module and $G$ is $\{Q\}$-filtered. For the terminology, see \cite{submitted}. 

Whenever $R$ is a right Ore domain, i.e., the class of regular elements is a right Ore set,
the class of strongly flat left $R$-modules is the class $^\bot \{Q^\bot\} $, where   $Q$ is the classical right quotient ring of $R$.

Several of our results about strongly flat modules are for modules over a nearly simple chain domain. Recall that a chain domain $R$, that is, a not-necessarily commutative integral domain for which the modules $R_R$ and $_RR$ are uniserial, is {\em nearly simple} if it has exactly three two-sided ideals, necessarily $R$, its Jacobson radical $J(R)$ and $0$. The reason why we concentrate on chain domains $R$ with classical quotient ring $Q$ is due to the fact that for these rings the $R$-module $_RK:=Q/R$ is uniserial, and thus, in the study of $\End(_RK)$, we can take advantage of our knowledge of the endomorphism rings of uniserial modules \cite{DungAlberto, DungFacchini, Facchinitransaction, Facsal, Pavel, puninsky}. In our Example~\ref{5.18}, we also take advantage of our knowledge of the endomorphism rings of cyclically presented modules over local rings \cite{Aminis}.

If $R$ is a right chain domain and the class of strongly flat $R$-modules is 
 covering, then 
$R$ is right invariant, that is, $aR = Ra$ for every $a \in R$. In this case, flat modules are strongly flat
(equivalently, the class $ \mathcal{SF}$ of strongly flat modules is closed under direct limit). 

We began this paper in September 2017, when both of us where visiting the Department of Algebra of Charles University in Prague, and continued in March 2018 when the first named author was visiting the IPM (Institute for Research in Fundamental Sciences) in Tehran. We are very grateful to both institutions  for their hospitality.

\section{The $R$-topology} \label{completions}

In Sections \ref{completions}, \ref{3} and \ref{4} of this paper, we suppose that {\em we have a ring $R$ and a multiplicatively closed subset $S$ of $R$ satisfying:} (1) {\em If $a,b\in R$ and $ab\in S$, then $a\in S$.} (2) {\em $S$ is a right Ore set in $R$.} (3) {\em The elements of $S$ are regular elements of~$R$}. (4) {\em The right ring of quotients $Q:=R[S^{-1}]$ of $R$ with respect to $S$ is a directly finite ring}. That is, our setting is that of \cite[Section~4]{submitted}. 

Correspondingly, we have  a Gabriel topology $\Cal G$ on $R$ consisting of all the right ideals $I$ of $R$ with $I\cap S\ne\emptyset$ (cf.~\cite[\S VI.6]{20}). In particular, the Gabriel topology  $\Cal G$ consists of dense right  ideals of $R$, the canonical embedding $\varphi \colon R\to Q:=R[S^{-1}]$ is an epimorphism in the category of rings, we view $R$ as a subring of $Q$ and $\varphi$ as the inclusion mapping, and ${}_RQ$ turns out to be a flat left $R$-module \cite[\S XI.3]{20}.
There is a hereditary torsion theory $(\Cal T,\Cal F)$ on $\Mod R$ in which
the torsion submodule of any right $R$-module $M_R$ consists of all the elements $x\in M_R$ for which there exists an element $s\in S$ with $xs=0$.  If we indicate the torsion submodule of $M$ by  
$t(M)$, then clearly $t(M) \otimes _R Q = 0$. A right $R$-module $M_R$ is in $\Cal F$, that is, is torsion-free, if and only if right multiplication $\rho_s\colon M_R\to M_R$ by $s$ is an abelian group monomorphism for every $s\in S$. Dually, we will say that a right $R$-module $M_R$ is {\em divisible} if right multiplication $\rho_s\colon M_R\to M_R$ by $s$ is an abelian group epimorphism for every $s\in S$, that is, if $Ms=M$ for every $s\in S$. Every homomorphic image of a divisible right $R$-module is divisible. If $A$ is a submodule of a right $R$-module $B_R$ and both $A_R$ and $B/A$ are divisible, then $B_R$ is divisible. Any sum of divisible submodules is a divisible submodule, so that every right $R$-module $M_R$ contains a greatest divisible submodule, denoted by $d(M_R)$. A right $R$-module $M_R$ is {\em reduced} if $d(M_R)=0$. For every module $M_R$, $M_R/d(M_R)$ is reduced.

We have that $\Cal G=\{\,I\mid I$ is a right ideal of $R$, and $\varphi(I)Q=Q\,\}$, and $\Cal G$ has a basis consisting of the principal right ideals $sR$, $s\in S$. Let $M_R$ be any right $R$-module.
By \cite [XI, Proposition 3.4] {20},
 the kernel of the canonical right $R$-module morphism $M_R\to M\otimes_R Q$ is
equal to $t(M)$. Note that if we set $K : = Q/R$, then 
$_RK_R$ is an $R$-$R$-bimodule and  $t(M_R)\cong\Tor_1^R(M_R, {}_RK)$ (see (15) and (16) in \cite[Section 3]{submitted}).

\bigskip

We now define a topology on any right $R$-module in the attempt of generalizing the $R$-topology studied by 
Matlis \cite{dim} for a commutative ring $R$. Our definition is as follows. Let $R$ be any ring with identity, not 
necessarily commutative, and $S$ be a subset of $R$ with the properties written at  the beginning of this section.
Given any right $R$-module $M_R$, the {\em $R$-topology} on $M_R$ has a 
neighborhood base of $0$ consisting, for every non-empty finite set of elements $s_1,\dots,s_n\in S$, of the 
submodules $$U(s_1,\dots,s_n):=\{\,x\in M_R\mid xR\subseteq Ms_1\cap\dots\cap Ms_n\,\}$$ of $M_R$. For the 
regular right module $R_R$, the 
$R$-topology on $R$ has a neighborhood base of $0$ consisting, for every non-empty finite set of elements $s_1,\dots,s_n\in S$, of the right ideals $$U(s_1,\dots,s_n):=\{\,x\in R\mid xR\subseteq Rs_1\cap\dots\cap Rs_n\,\}$$ of $R$.

\begin{Lemma} On the right $R$-module $R_R$, the right ideals $U(s)$ are two-sided ideals of $R$, $U(s)$ is the annihilator of the left $R$-module $R/Rs$, and the $R$-topology is a ring topology on $R$.\end{Lemma}

\begin{proof} Clearly, $U(s)=\{\,x\in R\mid xR\subseteq Rs\,\}$ is the annihilator of the cylic left $R$-module $R/Rs$, and hence $U(s)$ is a two-sided ideal. Moreover, $R$ is a right linearly topological ring \cite[p.~144]{20}, because every filter of two-sided ideals of a ring is a fundamental system of neighborhoods of $0$ for a right and left linear topology on the ring \cite[p.~144]{20}.\end{proof}

We will use $R_{R\Rtop}$ to denote the topological ring $R$ with the $R$-topology.

\begin{Lemma}\label{12.7} Every right $R$-module, with respect to its $R$-topology,  is a linearly topological module over the topological ring $R_{R\Rtop}$.\end{Lemma}

\begin{proof} It suffices to check property {\em TM}\,3 in \cite[p.~144]{20}. That is, we must prove that $(U_M(s):x)\supseteq U_R(s)$ for every $s\in S$, $x\in M_R$. Equivalently, that $xU_R(s)\subseteq U_M(s)$. Now if $r\in U_R(s)$, then $rR\subseteq  Rs$, so that $xrR\subseteq xRs\subseteq Ms$, i.e., $rx\in U_M(s)$.\end{proof}

\begin{Lemma}\label{Matlis} If the ring $R$ is commutative, the linear topology on any right $R$-module $M$ defined by the submodules $U(s)$, $s\in S$, coincides with the $R$-topology defined by Matlis in \cite{dim}.\end{Lemma}

\begin{proof} $U(s)=\{\,x\in M\mid xR\subseteq Ms\,\}=Ms$.\end{proof}

In the next proposition, we consider the behavior of continuity of right $R$-module morphisms when the modules involved are endowed with the $R$-topology. Recall that a submodule $M$ of a right $R$-module $N_R$ is an {\em RD-pure submodule} if $Mr=M\cap Nr$ for every $r\in R$ (equivalently, if the natural homomorphism $M\otimes R/Rr\to N\otimes R/Rr$ is injective for every $r\in R$, or if the natural homomorphism $\Hom(R/rR,N)\to\Hom(R/rR,M)$ is surjective for every $r\in R$.) See \cite[Proposition~2]{WarfPur}.

\begin{proposition}\label{easy} {\rm (a)} Every  right $R$-module morphism $f\colon M_R\to N_R$ between  two right $R$-modules $M_R$ and $N_R$  endowed with their $R$-topologies is continuous.

{\rm (b)} For every  right $R$-module $N_R$ and every $s\in S$, the $R$-submodule $U(s)$ of $N_R$ is the largest $R$-submodule of $N_R$ contained in $Ns$.

{\rm (c)} A submodule $M_R$ of a right $R$-module $N_R$  endowed with the $R$-topology is an open submodule of $N_R$ if and only if $M_R\supseteq U(s)$ for some $s\in S$.

{\rm (d)} A right $R$-module morphism $f\colon M_R\to N_R$ between  two right $R$-modules $M_R$ and $N_R$ with their $R$-topologies is an open map if and only if $f(M_R)\supseteq U(s)$ for some $s\in S$.

{\rm (e)} Every right $R$-module epimorphism $f\colon M_R\to N_R$ between  two right $R$-modules $M_R$ and $N_R$ is an open continuous map.

{\rm (f)} Every  right $R$-module isomorphism $f\colon M_R\to N_R$ is a homeomorphism when the  two right $R$-modules $M_R$ and $N_R$ are endowed with their $R$-topologies.

{\rm (g)} If $M_R$ is an RD-pure submodule of a right $R$-module $N_R$ and $M_R,N_R$ are endowed with their $R$-topologies, then the embedding $M_R\hookrightarrow N_R$ is a topological embedding.\end{proposition}

The proofs are easy and we omit them.

\section{The right $R$-module $\Hom(K_R, M\otimes_RK)$}\label{3}

In this section, the hypotheses on $R$ and $S$ are the same as in the previous section. 
For any right $R$-module $M_R$, we will be interested in the right $R$-module $$\Hom(K_R, M\otimes_RK).$$ Here the right $R$-module structure is given by the multiplication defined, for every $f\in \Hom(K_R, M\otimes_RK)$ and $r\in R$, by $(fr)(k)=f(rk)$ for all $k\in K$.

For any right $R$-module $M_R$, the right $R$-module $$\Hom(K_R, M\otimes_RK)$$ can be endowed with the $R$-topology, defined by the submodules $U(s_1,\dots,s_n):=U(s_1)\cap\dots\cap U(s_n)$ as a
neighborhood base of $0$. But we have that:

\begin{Lemma} For the modules $\Hom(K_R, M\otimes_RK)$, one has that $U(s)=V(s)$, where, for every element $s\in S$, $$V(s):=\{\,f\in \Hom(K_R, M\otimes_RK)\mid f((Rs^{-1})/R)=0\,\}.$$\end{Lemma}

\begin{proof} $(\subseteq)$. Let $f$ be an element of $U(s)$, so that $f\in\Hom(K_R, M\otimes_RK)$ and $fR\subseteq \Hom(K_R, M\otimes_RK)s$. In order to show that $f\in V(s)$ we have to prove that $f((Rs^{-1})/R)=0$. Fix $r\in R$. Then $fr=gs$ for some $g\in \Hom(K_R, M\otimes_RK)$. Hence $f(rs^{-1}+R)=(fr)(s^{-1}+R)=(gs)(s^{-1}+R)=g(ss^{-1}+R)=0$. Thus $f((Rs^{-1})/R)=0$. 

$(\supseteq)$. Suppose $f\in V(s)$, so that $f((Rs^{-1})/R)=0$. In order to prove that $f\in U(s)$, we must show that, for every fixed element $r\in R$, there exists $g\in \Hom(K_R, M\otimes_RK)$ with $fr=gs$. Define $g\colon K_R\to M\otimes_RK_R$ by $g(q+R)=f(rs^{-1}q+R)$ for all $q\in Q$. Then $g$ is a well defined right $R$-module morphism, because if $q\in R$, then $f(rs^{-1}q+R)=f(rs^{-1}+R)q\in f((Rs^{-1})/R)R=0$, and $fr=gs$.\end{proof}

\bigskip

We will denote by $V(s_1,\dots,s_n)$ the intersection $V(s_1)\cap\dots\cap V(s_n)$, but it is necessary to remark that:

\begin{Lemma} For every $s,s'\in S$, there exists $t\in S$ such that $V(s)\cap V(s')\supseteq V(t)$.\end{Lemma}

\begin{proof} 
Given $s,s'\in S$, there exist $t\in S$ and $r,r'\in R$ with $t=sr=s'r'$ \cite[Lemma~4.21]{goodwar}. Then  $s^{-1}=rt^{-1}$, so that $Rs^{-1}=Rrt^{-1}\subseteq Rt^{-1}$. Therefore $V(t)\subseteq V(s)$, because if $f\in \Hom(K_R, M\otimes_RK)$ and $f(Rt^{-1}/R)=0$, then $f(Rs^{-1}/R)=0$, that is, $f\in V(s)$.  Similarly, $V(t)\subseteq V(s')$.\end{proof}

A right (or left) $R$-module $M_R$ is {\em $h$-divisible} if every homomorphism $R_R\to M_R$ extends to an $R$-module morphism $Q_R\to M_R$ \cite[Section 2]{submitted}. Any right (or left) $R$-module $M$ contains a unique largest $h$-divisible submodule $h(M)$ that contains every $h$-divisible submodule of $M$. An $R$-module $M_R$ is {\em $h$-reduced} if $h(M_R)=0$, or, equivalently,  if $\Hom(Q_R,{}M_R)=0$ \cite{submitted}. 
Obviously, $h$-divisible right $R$-modules are divisible.

\begin{proposition} \label{equal} Divisible torsion-free right $R$-modules are 
$Q$-modules. 
 In particular, $h(M_R)=d(M_R)$ for any torsion-free right $R$-module~$M_R$.\end{proposition}

\begin{proof} Suppose $M_R$ torsion-free and divisible. Then right multiplication by $s$ is an automorphism of the abelian group $M$ for every $s\in S$. By the universal property of $Q=R[S^{-1}]$, the canonical ring antihomomorphism $R\to\End_\Z(M)$ extends to a ring antihomomorphism $Q\to\End_\Z(M)$ in a unique way. That is, there is a unique right $Q$-module structure on $M$ that extends the right $R$-module structure of $M_R$. Thus $M$ is a right $Q$-module. In particular, it is an $h$-divisible right $R$-module.
\end{proof} 

\bigskip

Let $M_R$ be a right $R$-module. For every element $x\in M_R$, there is a right $R$-module morphism $R_R\to M_R$, $1\mapsto x$. Tensoring with $_RK$, we get a right $R$-module morphism $\lambda_x\colon K_R\to M\otimes_RK$, defined by $\lambda_x(k)=x\otimes k$.  The canonical mapping $\lambda\colon M_R\to\Hom(K_R, M\otimes_RK)$, defined by $\lambda(x)=\lambda_x$ for every $x\in M_R$, is a right $R$-module morphism, as is easily checked. In the rest of this section, all $R$-modules are endowed with their $R$-topologies.

\begin{theorem}\label{righthreduced} Let $M_R$ be an $h$-reduced torsion-free right $R$-module. Then the canonical mapping $\lambda\colon M_R\to\Hom(K_R, M\otimes_RK)$ is an embedding of topological modules and $\Hom(K_R, M\otimes_RK)$ is complete.\end{theorem}

\begin{proof} The canonical mapping $\lambda\colon M_R\to\Hom(K_R, M\otimes_RK)$ is injective by \cite[Theorem~4.5]{submitted}. In order to show that $\lambda\colon M_R\to\Hom(K_R, M\otimes_RK)$ is an embedding of topological modules, it suffices to show that $\lambda^{-1}(V(s_1,\dots,s_n))=U(s_1,\dots,s_n)$ for every $s_1,\dots,s_n\in S$. Now $x\in \lambda^{-1}(V(s_1,\dots,s_n))$ if and only if $\lambda_x\in V(s_1,\dots,s_n)$, that is, if and only if $x\otimes (Rs_1^{-1}+\dots+Rs_1^{-1}/R)=0$ in $M\otimes_RK$. Equivalently, if and only if $x\otimes (rs_i^{-1}+R)=0$ in $M\otimes_RK$ for every $r\in R$ and $i=1,2,\dots,n$. By \cite[Step 3 of the proof of Theorem~4.5]{submitted}, this is equivalent to $xr\in Ms_i$  for every $r\in R$ and $i=1,2,\dots,n$, that is, if and only if $x\in U(s_1,\dots,s_n)$.

In order to prove that $\Hom(K_R, M\otimes_RK)$ is complete, we must show that every Cauchy net converges. Let $A$ be a directed set with order relation $\le$ and let $\{f_\alpha\}_{\alpha\in A}$ be a Cauchy net in $\Hom(K_R, M\otimes_RK)$. Define a morphism $f\in \Hom(K_R, M\otimes_RK)$ as follows. Since we are dealing with a Cauchy net, for every $s\in S$ there exists $\alpha\in A$ such that $f_\beta-f_\gamma\in V(s)$ for every $\beta,\gamma\in A$, $\beta,\gamma\ge\alpha$. Set $f(rs^{-1}+R)=f_\alpha(rs^{-1}+R)$ for every $r\in R$. We leave to the reader the easy verification that $f$ is a well defined mapping.
Let us check that $f(kr)=f(k)r$  for every $k\in K_R$ and $r\in R$. 
We have that $k=as^{-1}+R$ for some $a\in R$, $s\in S$. By the right Ore condition, there exist $r'\in R$ and $t\in S$ such that $as^{-1}r=r't^{-1}$. Since  $A$ is directed, there exists $\alpha$ such that 
$f ( r't^{-1} + R) = f_\alpha ( r't^{-1}+ R)$ and $f(  as^{-1}+R) r = f_\alpha ( as^{-1}+R) r $.
Therefore $f (kr) = f(k) r$.
It is now easily seen that $f$ is the limit of the Cauchy net.
\end{proof}

For any right $R$-module $M_R$ endowed  with its $R$-topology, the {\em (Hausdorff) completion} of $M_R$ is $\displaystyle \widetilde{M_R}:= \lim_{\longleftarrow} M/U(s_1,\dots,s_n)$. Notice that the set of all the submodules $U(s_1,\dots,s_n)$ of $M_R$ is downward directed under inclusion. Here $\{s_1,\dots,s_n\}$ ranges in the set of all finite subsets of $S$. There is a canonical mapping $\eta\colon M\to \widetilde{M_R}$, whose kernel is the closure $\overline{\{0\}}$ of $0$ in the $R$-topology of $M_R$. Clearly, $\overline{\{0\}}=\bigcap_{s_1,\dots,s_n\in S} U(s_1,\dots,s_n)=
\bigcap_{s\in S} U(s)=\{\,x\in M_R\mid xR\subseteq \bigcap_{s\in S}Ms\,\}$.

From Lemma \ref{12.7}, we get that if $M_R$ is a right $R$-module, the right $R$-module $\Hom(K_R, M\otimes_RK)$ with the topology defined by the submodules $V(s)$ is a topological module over the topological ring $R_{R\Rtop}$.

\begin{proposition} The right $R$-submodules $V(s)$ of the ring $\End(K_R)$ are two-sided ideals of $\End(K_R)$. The topology they define on $\End(K_R)$ is a ring topology. If $R$ is commutative, this topology on $\End(K_R)$ coincides with the topology on the completion $H$ of $R$ with respect to the $R$-topology \cite[p.~15]{dim}.\end{proposition}

\begin{proof} When we consider $M = R_R$, then, by 
  \cite[Step 2 of the proof of Theorem~4.5]{submitted},  the elements of $K$ annihilated by right multiplications of an element $s\in S$ are those of $Rs^{-1}/R$. It follows that $Rs^{-1}/R$ is a fully invariant submodule of $K_R$. From this we get that every $V(s)$ is a  two-sided ideal of the ring $\End(K_R)$. 

Every filter of two-sided ideals of a ring is a fundamental system of neighborhoods of $0$ for a right and left linear topology on the ring \cite[p.~144]{20}. Thus the topology defined by the two-sided ideals $V(s)$ is a ring topology on $\End(K_R)$. Moreover, if $R$ is commutative, the submodules $V(s)$ define the $R$-topology on the right $R$-module $\Hom(K_R, M\otimes_RK)$ for every module $M$ (Lemma~\ref{12.7}), which coincides with the $R$-topology defined by Matlis in \cite{dim} by  Lemma~\ref{Matlis}. 
Finally, Matlis' $R$-topology on $\End(K_R)$ coincides with the topology on the completion $H$ of $R$ with respect to the $R$-topology, because the topology on the completion $H$ coincides with the $R$-topology on $H$.
\end{proof}

\section{Torsion-free modules}\label{4}

In this section, we keep the same hypotheses and notations as in the previous two sections. 

As we have seen, for any right $R$-module $M_R$, there is a right $R$-module morphism $\lambda\colon M_R\to\Hom(K_R, M\otimes_RK)$, defined by $\lambda(x)=\lambda_x$ for every $x\in M_R$, where $\lambda_x\colon k\to x\otimes k$, and there is a canonical mapping $\eta\colon M\to \widetilde{M_R}$ of $M_R$ with its $R$-topology into its Hausdorff completion.

\begin{proposition} \label {torsionfreeright} Let $M_R$ be a torsion-free right $R$-module. Then: {\rm (a)} $\ker\lambda$ is the closure of $0$ in the $R$-topology; {\rm (b)} $\ker\lambda$ is the kernel of the canonical mapping $\eta\colon M\to \widetilde{M_R}$; and {\rm (c)} $\ker\lambda$ is equal to $h(M_R)$.\end{proposition}

\begin{proof} We have already remarked that the kernel of $\eta$ is the closure of $\overline{\{0\}}$ of $0$. Hence (a)${}\Leftrightarrow{}$(b).

The right $R$-module $\Hom(K_R,N_R)$ is $h$-reduced for every right $R$-module $N_R$ \cite[Theorem~2.8]{submitted}. Let $M_R$ be a torsion-free right $R$-module. Since $$\lambda\colon M_R\to \Hom(K_R, M\otimes_RK)$$ is a homomorphism into an $h$-reduced $R$-module, it follows that $h(M)\subseteq \ker\lambda$.

Let us prove that $\ker\lambda\subseteq \overline{\{0\}}$. Suppose $x\in \ker\lambda$. Then $x\otimes (rs^{-1}+R)$ is equal to zero in the tensor product $M\otimes K$. By \cite[Theorem~3.1(1)]{submitted}, there exists an element $y_{r,s}\in M_R$ such that $x\otimes rs^{-1}=y_{r,s}\otimes 1$ in $M\otimes_RQ$. Thus $xr\otimes 1=y_{r,s}s\otimes 1$ in $M\otimes_RQ$. Since $M_R$ is torsion-free, it follows that $xr=y_{r,s}s$ in $M_R$ by \cite[Theorem~3.1(1)]{submitted} again. This proves that $xR\subseteq \bigcap_{s\in S}Ms$, and so $\ker\lambda\subseteq \overline{\{0\}}$. 

Conversely, $\overline{\{0\}}\subseteq\ker\lambda$, because if $x\in \overline{\{0\}}$, then $xR\subseteq Ms$ for every $s\in S$, that is, for every $s\in S$ and every $r\in R$ there exists $m_{r,s}\in M$ with $xr=m_{r,s}s$. Then, for every element $rs^{-1}+R\in K$, we have that $x\otimes(rs^{-1}+R)=xr\otimes (s^{-1}+R)=m_{r,s}s\otimes (s^{-1}+R)=m_{r,s}\otimes s(s^{-1}+R)=0$ in $M\otimes_RK$. Thus $x\in\ker\lambda$. This proves that $\overline{\{0\}}=\ker\lambda$. Therefore (a) and (b) hold.

We now show that $\ker\lambda$ is divisible. For every $s\in S$, $s$ is invertible in $Q$, hence $sQ=Q$, so $sK=K$. Now if $x\in \ker\lambda$ and $t\in S$, then $x\in \overline{\{0\}}$, hence $x=yt$ for some $y\in M_R$. We must prove that $y\in \ker\lambda$, that is, that $y\otimes K=0$ in $M\otimes K$. But $y\otimes K=y\otimes sK=ys\otimes K=x\otimes K=0$ in $M\otimes K$. This proves that $\ker\lambda=\overline{\{0\}}$ is divisible. 
 Thus $ \ker\lambda = h(M)$ by Proposition \ref{equal}. 
\end{proof}

  Clearly, from Proposition~\ref{torsionfreeright}, we have that:
  
  \begin{corollary}
  	If $M_R$ is a torsion-free module, then $\widetilde{M_R} \cong \widetilde{M_R / h(M) }$.
  \end{corollary} 
  
\begin{lemma} \label{fourparts} Let $M$ be torsion-free right $R$-module. Then: 

{\rm (a)} Every element of $M\otimes_RK$ can be written in the form $x\otimes (s^{-1}+R)$ for suitable elements $x\in M_R$ and $s\in S$.

{\rm (b)} Let $s$ be an element of $S$. The elements $y$ of $M\otimes_RK$ such that $ys=0$ are those that can be written in the form $x\otimes (s^{-1}+R)$ for a suitable $x\in M_R$.

{\rm (c)} If $x\in M_R$, $r\in R$ and $s\in S$, then $x\otimes(rs^{-1}+R)=0$ in $M\otimes_RK$ if and only if $xr\in Ms$.

{\rm (d)} The set $\{\, U(s) \mid s \in S\,\} $ is downward directed. 
\end{lemma}

\begin{proof}
 In the proof of Steps 1, 2 and 3 of \cite [Theorem 4.5] {submitted}, we do not use the fact that $M$ is $h$-reduced. So the proofs of  (a), (b) and (c) are like those of Steps 1, 2 and 3 in \cite [Theorem 4.5] {submitted}. 

(d) Assume that $s, t \in S$. Then there exist $u \in S$ and $r_1,r_2\in R$ such that $ s ^{-1}= r_1u^{-1} $ and  
 $ t ^{-1}= r_2 u^{-1} $. If $m \in U(u)$ and $r \in R$, then $m \otimes (rs^{-1}+R) = 
m \otimes (r r_1u^{-1} +R)=0$. Part (c) implies that $m \in U(s)$, and so   $U(u) \subseteq U(s)$.
Similarly, $U(u) \subseteq U(t)$. 
\end{proof}

\begin{remark}\label{remarktorsionfree}
{\rm By Lemma~\ref{fourparts}(d), for 
 $M$ torsion-free, we have that $$\widetilde{M}= \lim_{\longleftarrow} M/U(s).$$ Notice that the kernel of 
the  canonical mapping $\eta \colon M\to \widetilde{M}$ is divisible by Theorem \ref{torsionfreeright}. } \end{remark}

\bigskip

Now let $M_R$ be a torsion-free right $R$-module, so that $$\lambda\colon M_R\to\Hom(K_R,\linebreak M\otimes_RK)$$ is continuous with respect to the $R$-topologies (Proposition~\ref{easy}(a)) and $\Hom(K_R,\linebreak M\otimes_RK)$ is Hausdorff. Notice that $M\otimes_RK$ and $M/h(M)\otimes_RK$ are isomorphic, so that $\Hom(K_R, M\otimes_RK)$ is complete (Theorem~\ref{righthreduced}). Thus $\lambda$ extends in a unique way to a continuous morphism $\widetilde{\lambda}\colon \widetilde{M} \to \Hom(K_R, M\otimes_RK)$. In Theorem~\ref{completion} and Example~\ref{quasismall}, we see { that $\widetilde{\lambda}$ is a continuous monomorphism, but not necessary an
 isomorphism.

\begin{theorem}\label{completion} Let $M_R$ be a torsion-free right $R$-module. Then there exists a right $R$-module monomorphism $\widetilde{\lambda}\colon \widetilde{M} \to \Hom(K_R, M\otimes_RK)$ such
 that $\lambda=\widetilde{\lambda}\eta$. 
\end{theorem}

\begin{proof} 
Define $\widetilde{\lambda}$ as follows. We know that $$\displaystyle\widetilde{M}= \lim_{\longleftarrow} M/U(s)\le\prod _{s\in S} M/U(s),$$ so that every element of $\widetilde{M}$ is of the form $\widetilde{m}=(m_{s}+U(s))_{s\in S}$.
 Set $\widetilde{\lambda}(\widetilde{m})(rs^{-1}+R)=m_{s}\otimes (rs^{-1}+R)$ for every $r\in R$, $s\in S$. 
In order to  prove that $\widetilde{\lambda}(\widetilde{m})\colon K_R\to M\otimes_RK$ is a well defined mapping and is $R$-linear, note first of all that  
if $s, t \in S$ are such that $U(t) \subseteq U(s)$ and $r \in R$, then $m_s - m_{t} \in U (s)$ implies that 
$m_s \otimes rs^{-1}+ R = m_{t} \otimes rs^{-1}+ R$ by Lemma~\ref{fourparts}(c). From this, it is  easily shown that $\widetilde{\lambda}$ is a well defined $R$-module morphism. Also notice that $\lambda=\widetilde{\lambda}\eta$.

Now we prove that  $\widetilde{\lambda}$ is a monomorphism.
Suppose that $\widetilde{m}=(m_{s}+U(s))_{s\in S}$ is in $\ker{\widetilde{\lambda}}$. Then, for any $k\in K$ and any $s\in S$ with $ks=0$, we have that $m_{s}\otimes k=0$ in $M\otimes_RK$. In particular, for every $r\in R$, $s\in S$, the identity $(rs^{-1}+R)s=0$ implies that $m_{s}\otimes (rs^{-1}+R)$ in $M\otimes_RK$. By Lemma~\ref{fourparts}(c), this means that $m_{s}r\in Ms$ for every $r$ and $s$. Hence $m_{s}\in U(s)$ for every $s\in S$. This shows that 
$\widetilde{\lambda}$ is injective.
\end{proof}
}

\begin{example} \label {quasismall}{\rm 
Let $R$ be the nearly simple chain domain in  \cite [Example 6.5] {chainringandprimeideal}. 
In that example, the $R$-module $Q/R$ can be chosen to be countably generated, because the
group $G$ is countable,  and so is its positive cone $P$. 
If  the skew field $K$ in that example  is countable, then $K[P]$
is countable. In order to construct the ring $R$, the authors consider a right and left Ore subset $S$ of 
 $K[P]$, which is necessarily countable because $K[P]$ is countable, and
 then they set $R := K[P]  S^{-1}$.
 Therefore if the skew field $K$ is countable, then 
$R$ is countable, and so $Q/R$ is a countably generated $R$-module. 
As $R_R$ is torsion-free, its completion is $\displaystyle\lim_{\longleftarrow} R/U(s)$ by 
Remark~\ref{remarktorsionfree}, and, for every non-zero element $s$ of $J(R)$, 
$U(s) = 0$ because $R$ is nearly simple. So $ \displaystyle R = \lim_{\longleftarrow} R/U(s)$. Let us prove that $R \ncong \End (K_R) $.
The module $K_R$ is a countably generated uniserial torsion locally coherent module (that is, every finitely generated submodule is coherent). By \cite[Proposition~8.1]{puninsky}, the module $K_R$ is not quasi-small. Since uniserial modules with a local endomorphism ring are quasi-small \cite{DungFacchini}, the ring $\End(K_R)$ cannot be isomorphic to $R$. 

The same argument applies to any nearly simple chain domain $R$ with $Q/R$ countably generated.}
\end{example}

%??? Here I would like to prove that $\widetilde{R_R}$ is the closure $\overline{\lambda(R_R)}$ of $\lambda(R_R)$ in $\End(K_R)$. Moreover the following four topologies coincide:

%(a) The topology induced on $\overline{\lambda(R_R)}$ by the $R$-topology of $\End(K_R)$.

%(b) The topology of $\widetilde{R_R}$ as an inverse limit.

%(c) The $R$-topology on $\widetilde{R_R}$.

%(d) The $\widetilde{R_R}$-topology on $\widetilde{R_R}$.

%We should also say that $\widetilde{R_R}$ consists of the elements $f\in\End(K_R)$ such that for every $s\in S$ there exists $r_s\in R$ such that $f(x)=r_sx$ for every $x\in Rs^{-1}/R$.

%???

\begin{proposition}
 If $R$ is a topological ring with a basis $B$ of neighborhoods of zero consisting of two-sided ideals, and 
$R/I$ is a local ring for every proper ideal $I\in B$, then the Hausdorff completion of $R$ is either $0$ or 
 a local ring.
\end{proposition}

\begin{remark}{\rm 
 The case of completion of $R$  equal to zero concernes only the trivial case of $B=\{R\}$. We will not consider this case in the proof.}
\end{remark}
\begin{proof}
Let $M_I $ be the maximal ideal of $R$ such that $M_I/I$
 is the maximal ideal of $R/I$ for every proper ideal  $I\in B$. If $I,J\in B$, then considering the canonical projection
$ R/I\cap J\to R/I$, one sees that 
 $M_{(I\cap J)}=M_I$. It follows that there exists a maximal ideal $M$ of
$ R$ such that $M_I=M$ for every proper ideal $ I\in B$. The completion of $R$ is the inverse limit of the rings 
$R/I$, which is a subring of the ring $\prod_{I\in B}R/I$, which has $\prod_{I\in B}M/I$ as a two-sided ideal, whose intersection $N$ with the inverse limit is a two-sided ideal of the inverse limit. Let us prove that the inverse limit is a local ring with maximal ideal $N$. It suffices to show that every element of the inverse limit not in $N$
 is invertible. Let $(x_I+I)_{I\in B}$ be an element in the inverse limit, but not in $N$. Thus
$x_I\in R$ and, for $I,J\in R$ with $I\subseteq J$, we have that $ x_I-x_J\in J$, i.e., $x_I+I $ 
is mapped to $ x_J+J $ via the canonical projection $R/I\to R/J.$ Also, $ x_I\notin M$ for some proper ideal $I$ of 
$B$. It follows that $x_I\notin M $ for every proper ideal $ I$ of $ B$. Thus $x_I+I\notin M/I,$
 hence is invertible in $R/M$. Let $y_I+I$ be the inverse of $x_I+I$ in $R/I$. Now the ring morphism 
$R/I\to R/J$ 
maps inverses to inverses. This shows that $(y_I+I)_{I\in B}$ is an element of the inverse limit, and concludes the proof.
\end{proof}

Therefore the completion of any local ring in the $R$-topology is a local ring.

\medskip

Note that, by Theorem \ref {completion} and  \cite [Proposition 2.6]{submitted}, if 
 $M_R$ is torsion-free, then $\widetilde{M_R}$ is torsion-free. %At the beginning of the next Section we will recall the definition of strongly flat modules, and some further properties of theirs, in a more general setting, but this is the right moment to prove the following result:

%\begin{corollary} Assume that $M$ is torsion-free h-reduced module. Then $\Ext (Q, M) = 0$ if and only if  $M \cong \Hom(K, M\otimes K)$.  \end{corollary}

%\begin{proof}
 %Note that $M/h(M)$ is h-reduced torsion-free and so we  have exact sequence
 %$ 0 \to M /h(M) \to \widetilde {M /h(M)} \to Ext^1(Q, M /h(M)) \to 0  $, by \cite [theorem 4.6] {submitted}.
%On the other hand since $h(M) $ is a $Q$-module,  $ Ext^1(Q, M /h(M)) \cong Ext^1(Q, M)  $. So by 
%considering the exact sequence $0 \to M /h(M) \to \widetilde {M} \to Ext^1(Q, M)  \to 0$ which comes 
%from Theorem \ref{hreduced} and five lemma, we have we have that $\widetilde{M_R} \cong \widetilde{M_R / h(M) }$.
%\end{proof}

\begin{theorem}\label{4.9} Let $R$ be a right Ore domain and $\widetilde{R_R}$ the completion of $R_R$ in the $R$-topology. Then $\widetilde{R_R}$ is a strongly flat right $R$-module.\end{theorem}

\begin{proof} We can apply the 
results of \cite[Section~3]{submitted}, which are right/left symmetric, that is, hold for both right $R$-
modules and left $R$-modules. Notice that $R_R$ is $h$-reduced. We have the short exact sequence\begin{equation}
\xymatrix{
0 \ar[r] & R_R \ar[r] & \End(K_R) \ar[r] & \Ext^1_R({}_RQ_R,R_R)  \ar[r] & 0.
}\end{equation} We know 
that $\widetilde{R_R}$ is a submodule of $\End(K_R)$ that contains $R_R$. Hence $\widetilde{R_R}/R_R$ is 
isomorphic to a submodule of $\Ext^1_R({}_RQ_R,R_R) $. In particular,
$\widetilde{R_R}/R_R$ is torsion-free, because $\Ext^1_R({}_RQ_R,R_R) $ is a $Q$-module, hence torsion-free. Let us 
prove that $\widetilde{R_R}/R_R$ is divisible, i.e., that $(\widetilde{R_R}/R_R)r=\widetilde{R_R}/R_R$ for every non-zero $r\in R$. Equivalently, we must prove that $\widetilde{R_R}\subseteq \widetilde{R_R}r+R_R$. Now $R_R$ is 
dense in $\widetilde{R_R}$, 
so that, for every $\widetilde{r}\in\widetilde{R_R}$ and every
 non-zero element $s$ of $R$, we have that $(\widetilde {r} + U(s)) \cap R_R \neq \emptyset. $ In particular, 
$(\widetilde{r}+U(r))\cap R_R\ne\emptyset. $ 
Notice that $U(r)\subseteq \widetilde{R_R}r$, because, for every $x\in U(r)$, we have that $xR\subseteq \widetilde{R_R}r$, hence $x\in \widetilde{R_R}r$. It follows that $(\widetilde{r}+\widetilde{R_R}r)\cap R_R\ne\emptyset. $ Thus there exists $\widetilde{r'}\in \widetilde{R_R}$ and $r''\in R_R$ with $
\widetilde{r}+\widetilde{r'}r=r''$. Therefore $\widetilde{r}=-\widetilde{r'}r+r''\in\widetilde{R_R}r+R_R$. This proves that $\widetilde{R_R}/R_R$ is divisible and torsion-free, hence a module over the division ring $Q$. Thus $\widetilde{R_R}/R_R\cong Q^{(X)}$ for some set $X$. The short exact sequence \begin{equation*}
\xymatrix{
0 \ar[r] & R_R \ar[r] & \widetilde{R_R}\ar[r] & Q^{(X)} \ar[r] & 0.
}\end{equation*} shows that $\widetilde{R_R}$ is strongly flat.\end{proof}

%??? To tell the truth, I would prefer if in the previous result, instead of the completion of R we prove it for the completion of any projective R-module, not only $R_R$. ???
\section{Strongly flat modules}

{\em In all this section, we consider two rings $R$ and $Q$, a bimorphism $\varphi \colon R\to Q$ in the category of rings, that is, $\varphi$ is both a monomorphism and an epimorphism, and we assume that ${}_RQ$ is a flat left $R$-module. For simplicity, we will view $R$ as a subring of $Q$ and $\varphi \colon R\to Q$ as the inclusion.} 

Let us recall some properties of such an inclusion $\varphi \colon R\hookrightarrow Q$. It is always possible to suppose $Q \subseteq Q_{\rm max}(R) $, the maximal ring of quotients of $R$ \cite[proof of Theorem XI.4.1]{20}. The 
inclusion $\varphi \colon R\to Q$ is an epimorphism in the category of rings if and only if the canonical $R$-$R$-bimodule morphism $Q \otimes _R Q \to Q$  induced by the multiplication $\cdot\colon Q\times Q\to Q$ of the ring $Q$ is an $R$-$R$-bimodule isomorphism \cite[Proposition~XI.1.2]{20}.
The family of all the subrings $Q$ of $Q_{\rm max}(R) $ with $\varphi \colon R\hookrightarrow Q$ a bimorphism and ${}_RQ$ flat is directed under inclusion \cite[Lemma XI.4.2]{20}. Its direct limit is the ``maximal flat epimorphic right ring of quotients'' 
$Q_{\rm tot}(R)$ of $R$ (see the paragraph after the proof of Corollary~\ref{12}).

By \cite [Theorem 4.8] {rep}, 
$\Ext^1(_RM,{} _RN) \cong \Ext^1(_QM,{} _QN) $ for any pair $M,N$ of  left $Q$-modules, and similarly for right $Q$-modules.

\begin{lemma}\label{Divisiblestronglyflat}
	Divisible strongly flat left $R$-modules are
	projective $Q$-modules.  
\end{lemma}
\begin{proof}
	Assume that $_RD$ is a divisible strongly flat module. Since $K \otimes Q = 0$, we have $K \otimes D = 0$. Since $_RD$ is flat, we have 
	$D \cong Q \otimes D$. For any exact sequence 
	$0 \to R^{(X)} \to D \oplus T  \to Q^{(Y)} \to 0  $, the corresponding exact sequence  $0 \to Q^{(X)} \to  D \oplus Q\otimes T  \to Q^{(Y)} \to 0  $ splits. Therefore 
	$D$ is a projective $Q$-module.
\end{proof}

Recall that any left perfect ring is directly finite. The following result shows that
when $_R \mathcal{SF}$ is covering, then $Q$ is left perfect. Thus the results is the same as in
 the commutative case, but the proof is necessarily different.

\begin{theorem} \label{Qisperfect}
	If all left $Q$-modules have a strongly flat cover as left $R$-modules,
	then
	$Q$ is left perfect.
\end{theorem}

\begin{proof}
	Assume that $_QM$ is a left $Q$-module and
	$f\colon{}_RS \to{} _RM$ is a strongly flat cover of $_RM$. Then we have
	an epimorphism $1\otimes f \colon Q\otimes S \to M$, $1\otimes f\colon
	q\otimes s\mapsto qf(s)$.
	Since $_RS$ is strongly flat, $_QQ\otimes_R S$ is a direct
	summand of a direct sum of copies of $Q$, i.e., 
	it is a projective left $Q$-module. Since projective left $Q$-modules
	are strongly flat left $R$-modules, the left $R$-module $_RQ\otimes S$ is
	strongly flat.
	But $f$ is a strongly flat precover of $M$, so that there exists $g \colon Q\otimes
	S \to S$ with
	$f g = 1\otimes f$.
	Note that $_RS$ is flat, and so $S$ can be embedded in $Q\otimes S$, that
	is, there is a left $R$-module monomorphism $h\colon _RS\to _RQ\otimes_R
	S$, defined by $h\colon s \mapsto 1\otimes s$.
	Then $f (gh) = f $, and thus $gh $ is an automorphism of $_RS$ because
	$f\colon _RS \to _RM$ is a cover. Thus $(gh)^{-1}gh=1$, so that
	$e:=h(gh)^{-1}g$ is an idempotent endomorphism of the left $R$-module
	$_RQ\otimes S$. Hence $e$ is an
	idempotent endomorphism of the left $Q$-module $_QQ\otimes S$. This shows that 
	$_QQ\otimes S$ is the direct sum of the image and the kernel of $e$,
	which are $Q$-modules. But the image of $e$ is the image of $h$. Hence
	the splitting monomorphism $h\colon s \mapsto 1\otimes s$ induces by
	corestriction a right $R$-module isomorphism of $_RS$ onto the
	$Q$-module $_Qh(S)$. By \cite[Section~2(7)]{submitted},
	 if a left $R$-module $_RA$ is a left $Q$-module $_QA$, then its
	unique left $Q$-module structure is given by the canonical isomorphism
	$\Hom(_RQ,_RA) \to{} _RA$. Therefore $S$ has a unique left $Q$-module
	structure, which extends its left $R$-module structure, and as such
	$_QS$ is a projective $Q$-module. Thus $f\colon{} _QS \to {}_QM$ is a left
	$Q$-module morphism.
	Note that projective $Q$-modules are strongly flat, and so $f\colon{} _QS \to{}
	_QM$ is a projective cover of $_QM$. Therefore 
	$Q$ is left perfect.
\end{proof}

The following result has a proof similar to that of  \cite[Proposition 2.4 ((1) and (2))]{silsal}.

\begin{lemma}\label{B-small}
	Let  $A$ be a module with a strongly flat cover and let
	 \begin{equation} 0 \to C \to M \to A \to 0  \label{(1)}\end{equation}
	be a special strongly flat precover of $A$. Then the exact sequence~{\rm (\ref{(1)})} is a strongly flat cover 
	if and only if   $C$ is $\mathcal{MC}$-small (i.e., $C + H = M$ and $C\cap H $ Matlis-cotorsion  imply
	$H = M$).
\end{lemma}

\begin{theorem} \label{twosidedidealIQ}
	Let
	$I$ be a two-sided ideal of $R$  such that $IQ = Q$.
	If all left $R/I$-modules have a strongly flat cover as left
	$R$-modules, then
	$R/I$ is left perfect.
\end{theorem}

\begin{proof}
	It is enough to show that every left $R/I$-module has a projective
	$R/I$-cover.
	Let $M$ be an $R/I$-module and $f\colon{} _RA \to {}_RM$ be a strongly flat
	cover of $_RM$.
	Since $IM = 0$, we have that $IA \subseteq \ker (f)$.
 Since ${}_RA$ is strongly flat, there exists an exact sequence
$ 0 \to R^{(X)} \to A  \oplus T \to Q^{(Y)}  \to 0 $, where $X$ and $Y$ are sets. 
 Since $IQ = Q$, we have $R/I \otimes Q = 0$. Thus we see that $A/IA$ is a projective left $R/I$-module. 
	So $f$ induces a map $h\colon A/IA \to M$, $h\colon a+ IA \mapsto f(a)$, and $\ker(h) = \ker(f)/IA$.
	Now
	$_RA$ is strongly flat and $IQ = Q$, and so $A/IA$ is a projective
	left $R/I$-module.
	We now show that $h$ is a cover for $M$ or, equivalently,  that 
	$\ker(f)/IA$ is small in  $A/IA$.
	Assume that $T + \ker(f) = A$, where $T$ is an $R$-module of $A$ that 
$IA \subseteq T $. Since $IQ = Q$, Hom$(Q, \ker(f)/ \ker(f) \cap T) = 0$. On the other hand, 
since $f\colon{} _RA \to {}_RM$ is  a strongly flat
	cover of $_RM$, the module $\ker(f)$ is Matlis-cotorsion by Wakamatsu Lemma (see \cite [Lemma 5.13]  {approx}), and thus
	$ \ker(f) \cap T$ is Matlis-cotorsion. Therefore $T = A$ by Lemma \ref{B-small}.
\end{proof}

\begin{lemma} \label{3.3}
	Assume that $R$ is a local ring with Jacobson radical $J$. Let
	$0 \to C \to S \to M \to 0$ be an $\mathcal{SF}$-cover for $M$. Then 
	$C \leq JS$.	
\end{lemma}

\begin{proof} Assume that $C \nleq JS$. Then $JS \neq S$. Since $R/J$ is a division ring, there exists a proper submodule $T/JS$ of $S/JS$ such that 
	$T/JS + (C + JS)/JS = S/JS$. Consequently $T + C  = S$.
	Consider the exact sequence 
	$0 \to T \cap C \to C \to S/ T \to 0$.
	Let us show that $\Hom(Q, S/T) = 0$. 
	Note that $R_R$ is essential in $Q_R$ (because $Q$ is a subring of $Q_{\rm max}(R)$). 
	Thus if $x \in Q \setminus R$,
	then  the right ideal of $I = \{\,r \mid xr \in R\,\}$ 
	is proper ideal of $R$, and so $I \leq J$. 
	By \cite [Part (b) of Theorem 3.9]{Goodearlnonsingular},
	$IQ = Q$ and so $JQ = Q$.
	If  Hom$(Q, S/T) \neq  0$, then there exists a proper submodule 
	$E$ of $Q$ such that $Q/E$ is isomorphic to a submodule of $S/T$.
	Thus $ Q = JQ \leq E$, which is a contradiction. Therefore 
	Hom$(Q, S/T) = 0$, and so $T\cap C \in Q^\bot$.
	Since $C$ is $\mathcal{MC}$-small, we have $T = S$, which is a contradiction.  
\end{proof}

It is known that if $R$ is commutative, $Q$ is the field of fractions of $R$ and
$_R \mathcal{SF}$ is covering, then $\pdim(_RQ) \leq 1$. We do not know what occurs in the 
non-commutative case. Therefore we now study the projective dimension of~${}_RQ$.  

\begin{proposition}\label{3.4'}  Suppose ${}_RQ$ is a projective left $R$-module. Then ${}_RQ$ is a finitely generated left $R$-module. \end{proposition}

\begin{proof} Since ${}_RQ$ is projective, it has a dual basis \cite[Exercise 11, pp.~202-203]{andersonfuller}, that is, there are elements $x_\alpha\in Q$ and morphisms $f_\alpha\colon {}_RQ\to {}_RR$ ($\alpha\in A)$, such that, for all $x\in Q$, $f_\alpha(x)\ne0$ for only finitely many $\alpha\in A$ and $x=\sum_{\alpha\in A}f_\alpha(x)x_\alpha$. Applying the functor ${}_QQ\otimes_R-\colon R\lMod\to Q\lMod$, we get left $Q$-module morphisms $1\otimes f_\alpha\colon {}_QQ\otimes_RQ\to {}_QQ\otimes_RR$. Now there are left $Q$-module isomorphisms $Q\to  {}_QQ\otimes_RQ$, $q\mapsto 1\otimes q$, and ${}_QQ\otimes_RR\to {}_QQ$, $q\otimes r\mapsto qr$. Composing, we get left $Q$-module endomorphisms ${}_QQ\to {}_QQ$, which are necessarily right multiplications by elements $y_\alpha\in Q$. Now, for all $x\in Q$, $f_\alpha(x)\ne0$ for only finitely many $\alpha\in A$. For $x=1$, we get that there is a finite subset $F$ of $A$ such that $f_\alpha(1)=0$ for every $\alpha\in A\setminus F$. Thus $(1\otimes f_\alpha)(1\otimes 1)=0$ for every $\alpha\in A\setminus F$. It follows that right multiplication by $y_\alpha$ maps $1$ to $0$, that is, $y_\alpha=0$  for every $\alpha\in A\setminus F$. It follows that $1\otimes f_\alpha\colon {}_QQ\otimes_RQ\to {}_QQ\otimes_RR$ is the zero mapping for every $\alpha\in A\setminus F$.  Thus $(1\otimes f_\alpha)(q\otimes q')$ is the zero element of ${}_QQ\otimes_RR$ for every $q,q'\in Q$. Hence $1\otimes f_\alpha(q')$ is the zero element of ${}_QQ\otimes_RR$. It remains to show that the mapping $_RR\to {}_QQ\otimes_RR$, $r\to 1\otimes r$, is injective, which is easily seen because $\Tor_1^R(K,R)=0$. This proves that $f_\alpha=0$ for every $\alpha\in A\setminus F$. As a consequence, $_RQ$ is isomorphic to a direct summand of ${}_RR^F$, so that ${}_RQ$ is a finitely generated left $R$-module. \end{proof}

\begin{corollary}\label{12} Let $R$ be a ring, $S$ a multiplicatively closed subset of regular elements of $R$, and suppose that $S$ is a right denominator set, so that the right ring of fractions $Q:=R[S^{-1}]$ exists. If ${}_RQ$  is a projective left $R$-module, then $Q=R$, that is, all the elements of $S$ are invertible in $R$. \end{corollary}

\begin{proof} By Proposition \ref{3.4'}, there are finitely many elements $r_1s_1^{-1},\dots,r_ns_n^{-1}$ that generate $Q$ as a left $R$-module. Reducing to the same denominator \cite[Lemma~4.21]{goodwar}, we find elements $r'_i\in R$ and $s\in S$ such that $s_ir'_i=s$ for every~$i$. Multiplying by $s^{-1}$ on the right and by $s_i^{-1}$ on the left, we get that $r'_is^{-1}=s_i^{-1}$. Thus $Q=\sum_{i=1}^n Rr_1s_1^{-1}\subseteq Rs_1^{-1}$. This proves that $Q=Rs_1^{-1}$. In particular, $s^{-2}\in Rs_1^{-1}$, from which $1\in Rs$. Let $t\in R$ be such that $1=ts$. Then $t=s^{-1}$ in~$Q$. Thus $Q=Rs_1^{-1}=Rt\subseteq R$, hence $Q=R$.\end{proof}

%Every ring $R$ has a maximal flat epimorphic right ring of quotients, unique up to isomorphism, denoted by $Q_{\rm tot}(R)$ \cite[Theorem~XI.4.1, 4.2]{20}.
On page 235 of \cite{20}, Stenstr\"om asks for necessary and sufficient conditions
for $Q_{\rm max}(R) $ to be equal to $Q_{\rm tot}(R)$.  He shows that if $Q_{\rm max}(R)$ is a 
right Kasch ring (i.e., a ring that contains a copy of its simple right modules), then $Q_{\rm max}(R) = Q_{\rm tot}(R)$.
If $R$ is right hereditary right noetherian \cite[Example 3, p.~235]{20} or
commutative noetherian
\cite[Example 4, p.~237]{20} or a right Goldie ring \cite[Theorem XII 2.5] {20}, then 
$Q_{\rm max}(R)$ is known to be Kasch.

\begin{example}
	{\rm Here is an example of a ring $R$ for which $Q_{\rm tot}(R) = Q_{\rm max}(R)$ is a projective right and left $R$-module, but $R \neq  Q_{\rm max}(R)$. Let $R$ be the ring of all lower  triangular $2 \times 2 $ matrices over a field $F$.
		The ring $R$ is  right nonsingular and $E(R_R)=S^{0}R=Q_{\rm max}(R)$ is a projective right and left $R$-module \cite [Exercise  14 on Page 78, and Corollary 2.31] {Goodearlnonsingular} (in Goodearl's notation, $S^{0}A := E(A/ Z(A_A))$, the injective envelope of $A/ Z(A_A)$ for any ring $A$). 
		More precisely, 
		$Q_{\rm max} (R) $ is the $2 \times 2 $ matrix ring over the field $F$, which is a semisimple artinian ring, hence a right and left Kasch ring, and 
		so $Q_{\rm max}(R) = Q_{\rm tot}(R)$ as we have seen above.}
\end{example}

We are now ready to consider the case of $\pdim(_RQ)\le 1$. Recall that a  cotorsion pair $(\Cal A,\Cal B)$ is said to be {\em hereditary} if $\Ext_R^i(A,B)=0$ for all $i\ge 1$, $A\in\Cal A$ and $B\in\Cal B$. 
Note that if $\Cal F$ is the class of flat modules and $\Cal E \Cal C$ the class 
of Enochs cotorstion modules, the the cotorsion pair 
 $(\Cal F,\Cal E \Cal C)$  is always a hereditary cotorsion pair. 
Similarly to \cite[Lemma 7.53] {approx}, we can show that:

\begin{lemma}
	The following conditions are equivalent for the pair of rings $R\subseteq Q$:
	
	{\rm (a)} $\pdim(_RQ) \leq 1$. 
	
		{\rm (b)} The cotorsion pair $(\mathcal{SF}, \mathcal{MC})$ is hereditary. 
	
\end{lemma}
\begin{proof}
	(a)${} \Rightarrow{}$(b).  Assume that $\pdim(_RQ) \leq 1$. Then strongly flat modules, which are summands of 
extensions of a direct sum of copies of $Q$ by a free module,  are of $\pdim$ at most $1$. Thus 
the cotorsion pair $(\mathcal{SF}, \mathcal{MC})$ is hereditary. 
	
	(b)${} \Rightarrow{}$ (a). By \cite [Theorem 3.5]{AS},  it  is enough to show that  that $\Ext^1(K, M)$ is $h$-reduced Matlis-cotorsion.
	Using the exact sequence $0 \to M \to E(M) \to E(M)/M \to 0$, we have the exact sequence 
$0 \to A \to B \to  \Ext^1(K, M) \to 0$, where 
	$A = $ Hom$(K, E(M))/$ Hom $(K, M)$ and $B = $ Hom $(K, E(M)/M)$.
	Note that, for every module $N$, Hom$(K, N)$ is Matlis-cotorsion and $h$-reduced by \cite [Theorem 2.8] {submitted}.
	So 	$\Ext^1(K, M)$ is $h$-reduced if and only if $A \in Q^\bot$. Now  $A \in Q^\bot$
	follows from the fact  that $(\mathcal{SF}, \mathcal{MC})$ is hereditary and the exact sequence
	$0 \to $ Hom $(K, M) \to $ Hom$(K, E(M)) \to A \to 0$.  As the module
$A$ is Matlis-cotorsion,  from the exact sequence $0 \to A \to B \to  \Ext^1(K, M) \to 0$ and the fact that 
  $(\mathcal{SF}, \mathcal{MC})$ is hereditary, we get that $\Ext^1(K, M)$ is Matlis-cotorsion. 
\end{proof}

As a consequence, $_R\mathcal{SF} ={} _R\mathcal{F}$ implies $\pdim(_RQ) \leq 1$. 

\begin{lemma}\label{stronglyflatreduced}
	Let $R$ be a right Ore domain and $Q $ the classical right quotient ring of $R$.  If $S$ is a strongly flat left $R$-module, then $S/h(S)$ is also strongly flat.
\end{lemma}
\begin{proof}
	Assume that $R$ is not a division ring. There exists an exact
	sequence $0 \to R^{(X)} \to S \oplus C \to Q^{(Y)} \to 0$.
	We claim that  Hom$(Q, R) = 0$. Otherwise, i.e., if $_RQ$ can be embeded in $_RR$, there exists a monomorphism $\varepsilon\colon _RQ\to _RR$. Then
	$\varepsilon$ can be viewed as a monomorphism $_RQ\to{} _RQ$. This monomorphism $\varepsilon$ is right multiplication by an element $q$ of $Q$. Now $\varepsilon$ a monomorphism implies $q\ne 0$, and $R$ right Ore domain implies $Q$ division ring.
	Hence $q$ is invertible in 
	$Q$, so that $R=Q$, which is a contradiction. This proves our claim. Now
	we have the embedding Hom$(Q, S \oplus C) \to $ Hom $(Q, Q^{(Y)})$.
	So we have an exact sequence $0 \to R^{(X)} \to (S\oplus C)/h(S\oplus C) \to Q^{(Y)}/ h(S\oplus C)\to 0$.
	Since $h(S\oplus C)$ is a torsion-free divisible module, it is a $Q$-module. But $Q$ is division ring, so 
	$h(S\oplus C)$ is a direct summand of $Q^{(Y)}$. It follows that $S /h(S)$ is strongly flat.	 
	
\end{proof}

 A  {\em left coherent}  ring is a ring over which every finitely generated left ideal is finitely presented or, equivalently, 
intersection of two finitely generated left ideals is finitely generated. 

\begin{theorem}\label{ideal}
Assume that  $R$ is a left coherent Ore domain with classical right quotient $Q$. A  left ideal ${}_RI$  of $R$ is a strongly flat left module if and only if $_RI$ is finitely generated  projective.
\end{theorem}
\begin{proof}
   Assume $_RI$ a non-zero strongly flat.
   We have the exact sequence of $R$-$R$-bimodules $0 \to R \to Q \to Q/R = K \to 0$.
   Since $_RI$ is flat, we get the exact sequence of left $R$-modules
   $0 \to R \otimes I \to Q \otimes I \to  K \otimes I \to 0$.
   Therefore $ K \otimes I \cong (Q \otimes I) / (R \otimes I)$. 
   We want to show $R/I$ embeds in $K \otimes I$ as $R$-modules. 
Consider the sequence of left $R$-modules
$0 \to{}_RI\to{}_RQ\to{}_RQ/I\to 0$ and apply to it the functor $Q\otimes_R-$. Since $Q_R$ is flat, we  get to
an exact sequence $0 \to Q\otimes_R I \to Q\otimes_R Q \to Q\otimes_R Q/I \to 0$.
Under the natural isomorphism $f\colon Q \otimes _R Q \to Q$, the image of $Q \otimes I $ is $QI = Q$, because  $I$ is non-zero, 
and the image of
$R\otimes I$ is $I$, and so $ K \otimes I  \cong (Q \otimes I)/( R\otimes I) \cong Q/I$ as a left $R$-module. Now $R/I \leq Q/I$ implies that  $R/I$ embeds in $K \otimes I$ as $R$-module.
   There  exists an exact
sequence $0 \to R^{(X)} \to I \oplus T \to Q^{(Y)} \to 0$.
   Since $K \otimes _R Q = 0$,
   we conclude that $K \otimes I$, and so $R/I$, embed in $K^{(X)}$ as left $R$-modules.
   Consequently, there exists an element $x \in{}  _R K^{(X)}$
   whose annihilator is equal to $I$.
    But the annihilator of an element of $ K^{(X)}$
   is equal to the intersection of finitely many annihilators of
   elements of $K$.
   If $ab^{-1} + R \in {}_RK$, then   ann$(ab^{-1} + R) = R \cap Rba^{-1}$.
   Note that $ R \cap Rba^{-1} \cong Ra \cap Rb$, which is a finitely generated left ideal of $R$ because $R$ is left coherent. Thus $I$ is a finitely generated left ideal of $R$ and, since it is flat,  $I$ is projective \cite [Theorem 4.30] {lam2}.
\end{proof}

\begin{lemma}\label{Sflat}
Let  $R$ be a right Ore ring with  classical right quotient ring $Q$. Then the strongly flat cover of any $h$-reduced flat left $R$-module is $h$-reduced.
\end{lemma}

\begin{proof}
Assume that $ M $ is a flat $h$-reduced module and $0 \to C \to S \to M \to 0$ is a strongly flat cover of 
$M$. Since $M$ is $h$-reduced, we can assume that $D:= h(C) = h(S)$.
So we have an exact sequence $ 0 \to C/D \to S/D \to M \to 0$.  By Lemma \ref{stronglyflatreduced}, $S/D$ is strongly flat. 
We can easily see  that this sequence is a strongly flat precover for $M$. 
Note that $C$ is torsion-free, and so $D$ is a left $Q$-module. Thus $\Ext_R^1(Q, D) = 0$. Since 
$C$ is  $\mathcal{MC}$-small in  $S$, we see that $C/D$ is
$\mathcal{MC}$-small. It follows that 
 $0 \to C/D \to S/D \to M \to 0$ is a strongly flat cover of $M$ by Lemma~\ref {B-small}(2), and so $S \cong S/D$. Therefore $D = 0$.
\end{proof}

\begin{proposition} \label{final}
Assume that $R$ is an Ore local domain with classical quotient ring $Q$. 
Suppose that $K \otimes _R S$ is 
 direct sum of copies of  $K$ for every strongly flat module $_RS$. 
If  $_R \mathcal{SF}$ is  a covering class, then  $_R \mathcal{SF} ={} _R \mathcal{F}$.
\end{proposition}

\begin{proof}
Firstly, notice that left $Q$-modules are 
injective as $R$-modules because $R$ is both a right and a left Ore domain. So, if $M$ is flat, then 
 $h(M)$ is 
a direct summand of  $M$, and therefore $M \cong h(M) \oplus M/h(M)$. 
Clearly, $Q$-modules are strongly flat, and thus it is enough  to show that any flat $h$-reduced module is strongly flat. 
Let $M$ be an $h$-reduced flat left module and $0 \to C \to S \to M \to 0$ be a strongly flat cover of 
$M$. By Lemma \ref{Sflat}, $S$ is also $h$-reduced, and thus  $C$ is an $h$-reduced flat left $R$-module. 
Assume that  $C \neq 0$, and  let  $0 \to C' \to S' \to C \to 0 $ be a strongly flat cover of $C$.
Then $S'$ is Matlis-cotorsion $h$-reduced strongly flat.
Note that by the left version of \cite [Theorem 4.6] {submitted},
  we have an exact sequence 
$0 \to S'  \to \Hom(K, K \otimes S') \to \Ext^1 (Q, S') \to 0$.  
Thus  $S'  \cong \Hom(K, K \otimes S')$.
Since $S'$ is strongly flat, it  is a direct summand of a direct sum of copies of $K$, 
and thus 
  $K \otimes S'$  is isomorphic to a direct summand of a direct sum of copies of $K$,   	 $K \otimes S' \cong K^{(Z)}$ say.
Thus $ S' \cong \Hom(K, K \otimes S') \cong  \Hom(K, K^{(Z)}) \cong \Hom (K, K \otimes R^{(Z)})$.
By \cite [Theorem 4.6] {submitted}, we have an exact sequence 
$0 \to R^{(Z)}  \to \Hom(K, K \otimes R^{(Z)}) \to \Ext^1 (Q, R^{(Z)}) \to 0$. 
Since $\Ext^1 (Q, R^{(Z)})$ is  a left $Q$-module,  
$R/J \otimes \Ext^1 (Q, R^{(Z)})  = 0$, where $J$ denotes the Jacobson radical of $R$.
 Consequently,  $JS' \neq S'$.
 On the other hand, 
by Lemma~\ref{3.3} and considering the pure exact sequence 
$0 \to C \to S \to M \to 0$, we see that $JC = C$. By 
Lemma \ref{3.3} and considering the exact sequence 
$0 \to C' \to S' \to C \to 0 $ again, we see that $J S' = S'$, which is a contradiction. This proves that $C = 0$, so that 
$M$ is strongly flat. 
\end{proof}

%Note that if $R$ is a chain domain, then $K$ is a uniserial module.  By Lemma \ref{localend} and \cite [Proposition 2.2] {DungAlberto}, every direct summand of copies of $K$ is isomorphic to a direct sum of copies of $K$. 

For any left module $_RM$, let $\Add(_RM)$ denote the class of all left $R$-modules isomorphic to direct summands of direct sums of copies of $_RM$.
 We will say that $\Add(_RM)$ is {\em trivial} if every direct summand of a direct sum of copies of $_RM$ is a direct sum of copies of $_RM$.

\begin{lemma}\label{Pavel}
Let $R$ be a nearly simple chain domain and let $_RK$ be the uniserial left $R$-module $Q/R$. Suppose $\Add(_RK)$ not trivial. Then 
there exists a submodule $V$ of $_RK$ that is not quasismall. Moreover, all the elements of $\Add(_RK)$ are isomorphic to $R$-modules of the form $_RK^{(X)} \oplus _RV^{(Y)}$.    
\end{lemma}
\begin{proof}
See  \cite[Theorem 1.1(ii)] {Pavel}. 
\end{proof}

In the next proposition, we describe uniserial strongly flat modules over Ore domains.

\begin{proposition}
	If $R$ is 
an Ore domain with  classical  quotient ring $Q$, then every non-zero uniserial strongly flat left module over  $R$ is isomorphic to $_RQ$ or~$_RR$.
\end{proposition}
\begin{proof}
 Let $_RU$ be a non-zero uniserial strongly flat left module over an Ore domain $R$. %, so that we have a short exact sequence $0\to{}_RR^{(X)}\to{}_RU\oplus{}_RC\to{}_RQ^{(Y)}\to 0$. Restricting this exact sequence to $_RU$, we find an exact sequence $0\to{}_RR^{(X)}\cap U\to{}_RU\to{}_RQ^{(Y)}.$
		%Step 2. R nearly simple chain domain. Then every flat left $R$-module is torsion-free, that is is an injective left R-module morphism.
Since $_RU$ is flat, considering the exact sequence   $0 \to R \to Q$, we have an embedding   $U\to Q\otimes_RU$. 
Hence the annihilator of every non-zero element of $ _RU$  is zero, and so cyclic submodules of $U$ are isomorphic to $_RR$. In particular, the ring $R$ is a left chain ring.
Moreover, $U$ is  the union of cyclic submodules isomorphic to
	$ _RR$, that is, a direct (linearly ordered) system of copies of $_RR$, where the connecting homomorphisms are right 
	multiplications by non-zero elements of $R$. Applying the functor $ _RQ\otimes_R-$, since tensor product commutes with direct limits, we get that $_RQ\otimes_RU$ is a direct limit of a direct system copies of $_RQ$, in which the 
	connecting isomorphisms are right multiplications by non-zero elements of $ R$, that is, the connecting isomorphisms are all left $R$-module automorphisms of $_RQ$. That is, $_RQ\otimes_RU \cong _RQ$.
	Hence $_RU$ embeds into $_RQ\otimes_RU\cong{}_RQ.$ If this embedding is onto, then $_RU\cong {}_RQ$. If the embedding is not onto, then $_RU$ is isomorphic to a proper submodule of $_RQ$, hence to a left ideal of $R$.
	By Theorem~\ref{ideal}, $_RU$ is cyclic, and so isomorphic to $_RR$. 
\end{proof}

\begin{lemma}\label{JV}
Let $R$ be a nearly simple chain domain with Jacobson radical $J$. If $\Add(K)$ is not trivial, 
 $V$ is as in Lemma {\rm \ref{Pavel}} and $M := \Hom (K, V^{(X)}),$ where $X$ is a  non-empty set, 
 then $JM \neq M$.
 That is, $M$ has maximal submodule. 
\end{lemma}

\begin{proof} The module $M=\Hom(_RK,{}_RV)$
 is a left $R$-module because  $_RK_R$ is a bimodule.
 Notice that $_RM$ always has a direct summand isomorphic to 
$\Hom(K, V)$, so that we can suppose that $X$ has exactly one element.
By \cite [(ii) of Theorem 1.1] {Pavel}, $K$ 
 has an endomorphism whose image is contained in $V$, say  $\varphi\colon {}_RK\to{}_RV $, that is injective but 
not surjective. Let us show that $\varphi$ is in $M$ but not in $JM$. 
 For every $j \in J$ and $\psi\in\Hom(_RK,{}_RV)$, the left $R$-module morphism $j\psi$
 is not injective. In fact, 
$ j\psi$ is right multiplication by $j$ viewed as a morphism $ _RK\to{} _RK $
composed with $\psi\colon{}_RK\to{}_RV. $ 
Thus the first morphism annihilates the element $j^{-1}+R,$
 so that the kernel of $ j\psi$  is non-zero. (This proves that $ j\psi$ is not injective for $j\ne 0.$
 But also when  $j=0$, $j\psi$ is not injective.)
Now every element of $JM $ a finite sum of elements of the form $j\psi,$
 i.e., of non-injective homomorphisms, hence is not injective because $_RK$ is
 uniserial, hence uniform.
Therefore  $\varphi\colon {}_RK\to{}_RV$ is not an element of  $JM$. 
\end{proof}

Recall that a two sided ideal $I$ of $R$ is {\em completely prime} if $xy \in I$ implies that $x \in I $ or $y \in I$ for every $x, y \in R$. 

\begin{theorem}
If $R$ is a right chain domain  with  classical right quotient ring $Q$ such that 
 $_R \mathcal{SF}$ is a covering class, then
 $R$ is  invariant and 
 $_R \mathcal{SF} = {}_R \mathcal{F}$. % or $R$ is a nearly simple chain domain such that $K$ is not countably generated.
\end{theorem}

\begin{proof}
If $I$ is a non-zero completely prime two-sided ideal of $R$, 
$R/I$ is a  left perfect domain by Theorem~\ref{twosidedidealIQ}, and so it is  is a division ring. Since $J(R)/I$ is an ideal of $R/I$, we conclude that 
 the only proper non-zero completely  prime ideal of $R$ is $J(R)$. 
A  chain domain $R$ is said to be  of rank one if $J(R)$ is its only non-zero completely prime ideal.
By \cite{chainringandprimeideal}, such a ring
 is either invariant, i.e., $aR = Ra$ 
for all $a \in R$, or it
is nearly simple, in which case $ 0$ and $J(R)$ are the only two-sided ideals,  or
$ R$ 
is exceptional
and there exists a non-zero prime ideal
$P$ 
properly contained in $J(R)$.  In this last case,  $\bigcap_n {P^n} = 0$ and 
there are no further ideals between $P$ and $J(R)$.
In the second and the third case, $J(R)$ is not neither right nor left finitely generated and $J^2 = J$. 
Now we break the proof in three steps. 

\medskip

{\em Step 1:  The ring $R$ cannot be exceptional. }

The Jacobson radical of $R/P$ is $J/P$, which cannot be nilpotent because $J^2 = J$.
 Thus $R/P$  cannot have a 
$T$-nilpotent Jacobson radical (see for example the proof of \cite [Lemma~3.33]{quasifereb}),   and so $R/P$ is neither a right nor a left perfect ring, and so the class of strongly flat left modules is not covering by Theorem~\ref{twosidedidealIQ}. 

\medskip

{\em Step 2: The ring $R$ cannot be nearly simple chain domain. }

Suppose $R$ a nearly simple chain domain. 
For every strongly flat module $_RS$, $K \otimes S$ is direct summand of a direct sum of copies of $K$, so that $K \otimes S$ belongs to $\Add(K)$. 
We have two cases:  
$\Add(K)$  is trivial or not. If $\Add(K)$  is trivial, then $_R \mathcal{SF}$ covering implies $_R \mathcal{SF} = {}_R \mathcal{F}$ by Proposition \ref{final}. But every cyclic (=\,finitely generated) ideal of $R$ is flat (= projective), so $_RJ$ must be flat, hence strongly flat (see for example \cite [Theorem 39.12(2)] {wis}). Thus 
$J$ must be finitely generated by Theorem~\ref{ideal},  which is a contradiction.  
Now assume that $\Add(K)$ is not trivial.  
By Lemma \ref{Pavel}, 
there exists a uniserial module $V$ which is not quasismall and every element in $\Add(K)$ is in form of $K^{(Y)} \oplus  V^{(X)}$ for suitable sets $X$ and $Y$. 
Let $0 \to C \to S \to J \to 0$ be a strongly flat cover of 
$J$. By Lemma \ref{Sflat}, $S$ is also $h$-reduced, and so  $C$ is an $h$-reduced flat left module. 
Assume $C \neq 0$, and  let  $0 \to C' \to S' \to C \to 0 $ be a strongly flat cover of $C$.
Then $S'$ is Matlis-cotorsion $h$-reduced strongly flat.
By the left version of \cite [Theorem~4.6] {submitted},
we have an exact sequence 
$0 \to S'  \to \Hom(K, K \otimes S') \to \Ext^1 (Q, S') \to 0$.  
So  $S'  \cong \Hom(K, K \otimes S')$.
Since $S'$ is strongly flat, $K \otimes S'$ is a 
direct summand of a direct sum of copies of $K$. 
Therefore there exist sets $X$ and $Y$ such that 
$K \otimes S' \cong K^{(Y)} \oplus  V^{(X)}$. 
So $ S' \cong \Hom(K, K \otimes S') \cong  \Hom(K, K^{(Y)})  \oplus \Hom(K, V^{(X)})$.
As we saw in the proof of  Theorem \ref{final},  if $Y$ is non-empty, we can consider 
the exact sequence $ 0 \to R^{Y} \to  \Hom(K, K \otimes R^{(Y)}) \to \Ext^1_R ( Q,  R^{(Y)}) \to 0 $, 
and conclude that  
$J  \Hom(K, K^{(Y)}) \neq  \Hom(K, K^{(Y)})$. Similarly,
by  Lemma \ref {JV}, if $X$ is non-empty, $J \Hom(K, V^{(X)}) \neq \Hom(K, V^{(X)})$.
Consequently,  $JS' \neq S'$.
By Lemma \ref{3.3}, considering the pure exact sequence 
$0 \to C \to S \to J \to 0$, we see that $JC = C$. By Lemma \ref{3.3} again, from the exact sequence 
$0 \to C' \to S' \to C \to 0 $, we get that $J S' = S'$, which is a contradiction. This proves that $C = 0$, so that 
$J$ is strongly flat, which contradicts Theorem~\ref{ideal}.  

\medskip

{\em Step 3: The ring $R$ is invariant and $_R \mathcal{SF} = {}_R \mathcal{F}$.}

By Steps 1 and 2,  the ring  $R$ must be   invariant. Therefore the endomorphism ring of every uniserial module is local
 (the proof is similar to the commutative case, because, like in the proof of  \cite [Corollary 3] {Facsal},
every uniserial module is unshrinkable,  and so the endomorphism ring of every uniserial module is local like in the proof of 
\cite [Example  2.3(e)] {Facchinitransaction}).  Thus
$\End (K _R)$ is  local and  every direct summand of copies of $K$ is isomorphic to a direct sum of copies of $K$ because $K_R$ is uniserial by \cite [Proposition 2.2] {DungAlberto}. Thus $_R \mathcal{SF} ={} _R \mathcal{F}$ by  Proposition~\ref{final}. \end{proof}

We conclude with an example concerning right noetherian right chain domains. In a right noetherian right chain domain $R$, all right ideals are principal and two-sided \cite[Lemma~3.2]{bessen}. In particular, $J(R)=pR$ for some $p\in R$. The right noetherian right chain domain $R$ is said to be {\em of type $\omega$} \cite[p.~26 and Lemma 3.4]{bessen} if its chain of right ideals (=\,two-sided ideals) is the chain $$R=p^0R\supset J(R)=pR\supset p^2R\supset\dots\supset 0=\bigcap_{n\ge0}p^nR.$$ Thus for every non-zero right ideal $I$ of $R$, we have that $\End (R_R/I)\cong R/I$ is a right artinian ring, hence a perfect ring. In the next example, we show that this is also true for every non-zero principal left ideal $I$ of a right noetherian right chain domain $R$ of type $\omega$ which is not left Ore. Notice that in our example of right noetherian right chain domain of type $\omega$ which is not left Ore, the ring is not left chain (otherwise it would be left Ore) and is not left noetherian \cite[Proposition 3.7]{bessen}. The main example of such a ring can be constructed with the skew poynomial ring with coefficients in a field $F$, where $F$ has an endomorphism that is not an automorphism of $F$.

\begin{example}\label{5.18}
{\rm Let $R$ be a right noetherian right chain domain of type $\omega$ which is not left Ore.
For every non-zero principal left ideal $I$ of $R$, the endomorphism ring $\End (_RR/I)$ is a perfect ring. }
\end{example}

\begin{proof}
For every non-zero element $x\in R$, we have that $xR=p^nR$  for some $n\ge 0$. 
Therefore $x=p^nu$ for some invertible element $u\in R$.  Right multiplication by $u$
 induces an isomorphism $R/Rp^n\to R/Rx$. Hence it suffices to show that 
 $\End(R/Rp^n)$ is right and left perfect for $n\ge 1$.
 Notice that $ Rp^n\subseteq p^nR$. Set
$S:=\End(R/Rp^n)\cong E/Rp^n$, where $E: = \{\, r \in R \mid p^n r \in R p^n \,\}$ denotes the idealizer of $Rp^n$ in $R$, and set $K : = \{\, r \in R \mid p^n r \in J p^n \,\}$. By \cite [Theorem 2.1]{Aminis}, $S$ has at most two maximal ideals, the ideals $K/Rp^n$ and $(J \cap E)/ Rp^n $.
Let us show that $ K \subseteq J$.
Assume the contrary, so that $K$ contains a unit $u$ of $R$. Therefore 
 $p^n u = r p^n$ for some  $r\in J$. 
Then $ r = p^j v $ for some unit $v$ of $R$ and some $j \geq 1$. 
Thus $p^n = p^j v p^n u^{-1}$.
If $j \geq n$, then $ 1 = p^{j -n } v p^n u^{-1}$, which implies $J(R)=R$, a contradiction. 
If  $j < n$, then $p^{n -j} = v p^n u^{-1}$. Thus $p^{n -j}$ belongs to the two-sided ideal $pR$ of $R$, which is a contradiction because $n-j<n$. 
%R = v p^n R$ and since $p^n R$ is a two sided ideal, we get  $p^{n -j}R \leq p^n R $,
%which is a contradiction.  
Therefore $ K \subseteq J$, so 
  $S$ is local with maximal ideal $(J \cap E)/ Rp^n $.
  We claim that  if $p^n y \in E$, then  $p^ny \in Rp^n$.  
To prove the claim, assume that  $p^n y \in E$. Then there exists $s\in R$ such that $p^n p^n y = sp^n$. Similarly, there exists $i \geq 0$ and a unit $u$ in $R$ such that 
$s = p^i u$. If $i \geq n$, then we are done, the claim is proved. Otherwise, if $i < n$, by supposing that $y = p^j v$ for some unit $v$, we get that 
$ u^{-1} p^l v = p^n$, so $l > n$, which is a contradiction by \cite [Lemma 2.3] {Aminis}. 

Therefore $(J\cap E)^n\subseteq J^n\cap E=p^nR\cap E\subseteq Rp^n$, to that the Jacobson radical $J(S)=(J\Cal E)/Rp^n$ of the local ring $S$ is nilpotent. It follows that $S=\End(R/Rp^n)$ is a right and left perfect ring.
 % 
% Hence $S$ is right or left perfect if and only if $J(S)$ is left or right T-nilpotent. 
%To show that $J(S)$ is neither right nor left T-nilpotent, at first we prove that if $p^n y \in E$, then  $p^ny \in Rp^n$.  
%Assume that  $p^n y \in E$, so there exists $s\in R$ such that $p^n p^n y = sp^n$. Also there exists $i \geq 0$ and $u$ unit in $R$ such that 
%$s = p^i u$. If $i \geq n$, then we have done. Otherwise, by assuming that $y = p^j v$ for some unit $v$, we see that 
%$ u^{-1} p^l v = p^n$, where $l > n$, which is a contradiction by \cite [Lemma 2.3] {Aminis}. 
%
%Now assume that $a_1 , a_2, \cdots \in E\cap J \setminus   Rp^n$. 
%Since $a_i \notin J$,  there exists $j_i \geq 1$ such that $a_i = p^{j_i} u_i$, for some unit $u_i$. 
%Also since  $Rp^i \subseteq p^i R$, we see that $ r  = a_1 a_2 \dots a_n = p^n x$, for some $x \in R$.  
%Since $r \in E$ as we said, we conclude that $r \in Rp^n$ and similarly we see that $a_n \cdots a_1 \in Rp^n$. 
%This proves that the Jacobson radical of $\End R/p^n$ is right and left $t$-nilpotent. 
\end{proof}

%???? In the above example what can we say about Endomorphism ring of $R/I$, where $I$ is left ideal of $R$. Is there any relation between End of $R/I$ and end of $R/Rx$, where $x$ is a non-zero  element of $I$? 

\end{document}